\documentclass[twoside,10pt]{amsart}

\usepackage{amsthm,amsmath,amsfonts,epsfig,amssymb,amscd}
\usepackage[T1]{fontenc}
\usepackage[utf8]{inputenc}
\usepackage{graphicx}
\usepackage{enumerate}
\usepackage{xy}
\usepackage{graphics}
\usepackage{bbm}
\usepackage[colorlinks=true]{hyperref} 
\usepackage[margin=1in]{geometry}

\xyoption{all}
\entrymodifiers={+!!<0pt,\fontdimen22\textfont2>}

\usepackage{color}

\newtheorem{thm}{Theorem}[section]

\newtheorem{prop} [thm]{Proposition}
\newtheorem{lem} [thm]{Lemma}
\newtheorem{coro}[thm]{Corollary}

\newtheorem*{thm*}{Theorem}
\theoremstyle{definition}
\newtheorem{defin}[thm]{Definition}

\theoremstyle{remark}
\newtheorem{rem}[thm]{Remark}
\newtheorem{exem}[thm]{Example}

\def\Z{\mathbb{Z}}
\def\N{\mathbb{N}}
\def\A{\mathbb{A}}
\def\OO{\mathcal{O}}

\def\1{\mathbbm{1}}

\def\id{\mathrm{id}} 

\newcommand{\aone}{\A^1}
\newcommand{\mot}{\aone}
\newcommand{\pone}{\mathbb{P}^1}
\def\gm{\mathbb{G}_m} 

\def\spec#1{\mathrm{Spec}(#1)}

\def\sKM{\mathbf{K}^{\mathrm{M}}} 
\def\sKMW{\mathbf{K}^{\mathrm{M\hspace{-.2ex}W}}} 

\newcommand{\tch}[1]{%
  \mathchoice{\widetilde{\mathrm{CH}}^{\raisebox{-.5ex}{$\scriptstyle#1$}}}
             {\widetilde{\mathrm{CH}}^{\raisebox{-.8ex}{$\scriptstyle#1$}}}
             {}
             {}
}
\newcommand{\tchi}[2]{%
  \mathchoice{\widetilde{\mathrm{CH}}_{#2}^{\raisebox{-.5ex}{$\scriptstyle#1$}}}
             {\widetilde{\mathrm{CH}}_{#2}^{\raisebox{-.7ex}{$\scriptstyle#1$}}}
             {}
             {}
}
\def\ch#1#2{\tch{#1}(#2)}

\def\chst#1#2#3#4{\tchi{#1}{#2}(#3,#4)}

\def\H{\mathrm{H}}

\def\tZ{\tilde\Z}

\def\cor#1{\widetilde{\mathrm{Cor}}_{#1}}

\def\Zar{\mathrm{Zar}}
\def\Nis{\mathrm{Nis}}

\def\sh#1#2{\widetilde{\mathrm{Sh}}_{#2,#1}}

\def\Cstar{\mathrm{C}^{\mathrm{sing}}_*}
\def\CstarS{\mathrm{C}^{\mathrm{0}}_*}

\DeclareMathOperator{\Der}{D}
\newcommand{\DAe}{\Der_{\A^1}^{\mathrm{eff}}(k,R)}

\newcommand{\DMte}{\widetilde{\mathrm{DM}}{}^{\mathrm{eff}}\!(k,\Z[1/e])}
\newcommand{\DMteR}{\widetilde{\mathrm{DM}}{}^{\mathrm{eff}}\!(k,R)}
\newcommand{\DMteZ}{\widetilde{\mathrm{DM}}{}^{\mathrm{eff}}\!(k)}
\newcommand{\DMe}{\mathrm{DM}^{\mathrm{eff}}(k,R)}

\newcommand{\DA}{\Der_{\A^1}(k,R)}
\newcommand{\DMt}{\widetilde{\mathrm{DM}}(k)}
\newcommand{\DMtR}{\widetilde{\mathrm{DM}}(k,R)}
\newcommand{\DMtinv}{\widetilde{\mathrm{DM}}(k,\Z[1/e])}
\newcommand{\DM}{\mathrm{DM}(k,R)}
\newcommand{\SH}{\mathrm{SH}(k)}
\newcommand{\SHS}{\mathrm{SH}^{S^1}(k)}
\newcommand{\Spt}{\mathrm{SH}}

\newcommand{\eff}{{\mathrm{eff}}}
\DeclareMathOperator*{\colim}{colim}
\DeclareMathOperator*{\hocolim}{hocolim}
\newcommand{\heart}{\heartsuit}

\newcommand{\derL}{\mathbf{L}}

\begin{document}

\title{On the effectivity of spectra representing motivic cohomology theories}
\author{Tom Bachmann \and Jean Fasel}
\date{\today}

\begin{abstract}
Let $k$ be an infinite perfect field. We provide a general criterion for a spectrum $E\in \SH$ to be effective, i.e. to be in the localizing subcategory of $\SH$ generated by the suspension spectra $\Sigma_T^{\infty}X_+$ of smooth schemes $X$.
As a consequence, we show that two recent versions of generalized motivic cohomology theories coincide.
\end{abstract}

\maketitle

\pagenumbering{arabic}



\section*{Introduction}

In \cite{bachmann-very-effective}, the first author undertook the study of the very effective slice spectral sequence of Hermitian $K$-theory, which could be seen as a refinement of the analogue in motivic homotopy theory of the famous Atiyah-Hirzebruch spectral sequence linking singular cohomology with topological $K$-theory. He observed that the generalized slices were 4-periodic and consisting mostly of well understood pieces, such as ordinary motivic cohomology with integral and mod 2 coefficients. However, there is a genuinely new piece given by a spectrum that he called \emph{generalized motivic cohomology}. Thus, Hermitian $K$-theory can be ``understood'' in terms of ordinary motivic cohomology and generalized motivic cohomology in his sense. Even though he was able to deduce abstractly some properties for this motivic cohomology, some questions remained open.

On the other hand, different generalizations of ordinary motivic cohomology recently appeared in the literature, always aimed at understanding better both the stable homotopy category of schemes and its ``abelian'' version. First, Garkusha-Panin-Voevodsky developed the formalism of framed correspondences and its linear version. Among many possible applications, this formalism allows to define an associated motivic cohomology, the first computations of which were performed in \cite{Neshitov14}. Second, Calm\`es-D\'eglise-Fasel introduced the category of finite MW-correspondences and its associated categories of motives (\cite{Calmes14b,Deglise16}) and performed computations allowing to recast most of the well-known story in the ordinary motivic cohomology in this new framework. Third, Druzhinin introduced the category of GW-motives (\cite{Druzhinin17a}) producing yet another version of motivic cohomology.

This flurry of activity leads to the obvious question to know the relations between all these theories, paralleling the situation at the beginnings of singular cohomology. This is the question we address in this paper with a quite general method. To explain it, note first that all these motivic cohomologies are represented by ring spectra in the motivic stable homotopy category (of $\mathbb{P}^1$-spectra) $\SH$. This category is quite complicated, but the situation becomes much better if the ring spectra are in the localising subcategory $\SH^{\eff}$ generated by the image of the suspension spectrum functor $\Sigma_T^{\infty}:\SHS\to \SH$. This category is endowed with a $t$-structure (\cite[Proposition 4]{bachmann-very-effective}) whose heart is much easier to understand than the heart of the (usual) $t$-structure of $\SH$. Moreover, many naturally occurring spectra turn out to be in this heart. Thus, our strategy is to prove that the relevant spectra are in $\SH^{\eff}$, or \emph{effective}, then show that they are represented by objects in the heart, and finally compare them via the natural maps linking them. Unsurprisingly, the first step is the hardest and the main part of the paper is devoted to this point. The criterion we obtain is the following (Theorem \ref{thm:eff-crit}).

\begin{thm*}
Let $E \in \SH$, where $k$ is a perfect field. Then $E \in \SH^\eff$ if and only if for all $n \ge 1$ and all finitely generated fields $F/k$, we have $(E \wedge \gm^{\wedge n})(\hat \Delta^\bullet_F) \simeq 0$.
\end{thm*}

In the statement, $\hat \Delta^\bullet_F$ denotes the essentially smooth cosimplicial scheme whose component in degree $n$ is the semi-localization at the vertices of the standard algebraic $n$-simplex over $F$. Making sense of $(E \wedge \gm^{\wedge n})(\hat \Delta^\bullet_F)$ requires some contortions which are explained in Section \ref{sec:general-criterion}. The appearance of $\hat \Delta^\bullet_F$ is explained by the need to compute the zeroth (ordinary) slice of a spectrum, using Levine's coniveau filtration (\cite{levine2008homotopy}).

Having this criterion in the pocket, the last two (much easier) steps of our comparison theorem take place in the proof of our main result (Theorem \ref{thm:comparison}).

\begin{thm*}
Let $k$ be an infinite perfect field of exponential characteristic $e \ne 2$ and let 
\[
M: \SH \leftrightarrows  \widetilde{\mathrm{DM}}(k): U
\]
be the canonical adjunction. Then the spectrum $U(\1)$ representing MW-motivic cohomology with $\Z$-coefficients is canonically isomorphic to the spectrum $H\tZ$ representing abstract generalized motivic cohomology with $\Z$-coefficients.
\end{thm*}

The organization of the paper is as follows. We briefly survey the main properties of the category of MW-motives, before proving in Section \ref{sec:rational} that the presheaf represented by $\gm^{\wedge n}$ is rationally contractible (in the sense of \cite[\S 2]{Suslin03}) for any $n\geq 1$. Unsurprisingly, our proof follows closely Suslin's original method. However, there is one extra complication due to the fact that the presheaf represented by $\gm^{\wedge n}$ is in general not a sheaf. We thus have to compare the Suslin complex of a presheaf and the one of its associated sheaf in Section \ref{sec:semilocal}. This part can be seen as an extension of the results in \cite[\S 4]{Fasel16b} to the case of semi-local schemes, i.e. localizations of a smooth scheme at a finite number of points. The proof of our criterion for effectivity takes place in the subsequent section. Finally, we prove our comparison result in Section \ref{sec:MWmotives}, where all the pieces fall together. 

In the last few paragraphs of the article, we give some examples of applications of our results, one of them being a different way to prove the main result of \cite{Suslin03} avoiding polyrelative cohomology. 

\subsection*{Conventions}
Schemes are separated and of finite type over a base field $k$, assumed to be infinite perfect of characteristic different from $2$.
Recall that a field $k$ is said to have exponential characteristic $e=1$ if $\mathrm{char}(k) = 0$, and $e = \mathrm{char}(k)$ else.

\subsection*{Acknowledgments}

The first author would like to thank the Hausdorff research institute for mathematics, during a stay at which parts of these results where conceived. Both authors would like to thank the Mittag-Lefler Institute for a pleasant stay, where some problems related to the present paper were discussed. The authors would like to thank Maria Yakerson for comments on a draft.


\section{Recollections on MW-correspondences}\label{sec:recollections}

In this section, we briefly survey the few basic features of MW-correspondences (as constructed in \cite[\S 4]{Calmes14b}) and the corresponding category of motives (\cite[\S 3]{Deglise16}) that are needed in the paper. Finite MW-correspondences are an enrichment of finite correspondences after Voevodsky using symmetric bilinear forms. The category whose objects are smooth schemes and whose morphisms are MW-correspondences is denoted by $\cor k$ and we have a sequence of functors
\[
\mathrm{Sm}_k\stackrel{\tilde\gamma}\to \cor k\stackrel{\pi}\to \mathrm{Cor}_k
\]
such that the composite is the classical embedding of the category of smooth schemes into the category of finite correspondences. For a smooth scheme $X$, the corresponding representable presheaf on $\cor k$ is denoted by $\tilde{\mathrm{c}}(X)$. This is a Zariski sheaf, but not a Nisnevich sheaf in general (\cite[Proposition 5.11, Example 5.12]{Calmes14b}). The associated Nisnevich sheaf also has $\mathrm{MW}$-transfers (i.e. is a (pre-) sheaf on $\cor k$) by \cite[Proposition 1.2.11]{Deglise16} and is denoted by $\tilde{\Z}(X)$.

Consider next the cosimplicial object $\Delta^\bullet_k$ in $\mathrm{Sm}_k$ defined as usual (see \cite[\S 6.1]{Calmes14b}). Taking the complex associated to a simplicial object, we obtain the Suslin complex $\Cstar \tilde\Z(X)$ associated to $X$, which is the basic object of study. Applying this to $\gm^{\wedge n}$, we obtain complexes of Nisnevich sheaves $\tilde\Z\{n\}$ for any $n\in \N$ and complexes $\tilde\Z(n):=\tilde\Z\{n\}[-n]$ whose hypercohomology groups are precisely the MW-motivic cohomology groups in weight $n$. In this paper, we will also consider the cosimplicial object $\hat\Delta^\bullet_k$ obtained from $\Delta^{\bullet}_k$ by semi-localizing at the vertices (see \cite[5.1]{levine2008homotopy}, \cite[paragraph before Proposition 2.5]{Suslin03}). Given a finitely generated field extension $L$ of the base field $k$, the same definition yields cosimplicial objects $\Delta^{\bullet}_L$ and $\hat\Delta^\bullet_L$ that will be central in our results.
If $L/k$ is separable, then note that both $\Delta^{\bullet}_L$ and $\hat\Delta^\bullet_L$ are simplicial essentially smooth schemes.

The category $\cor k$ is the basic building block in the construction of the category of effective MW-motives (aka the category of MW-motivic complexes) $\widetilde{\mathrm{DM}}^{\eff}(k)$ and its $\pone$-stable version $\widetilde{\mathrm{DM}}(k)$ (\cite[\S 3]{Deglise16}). The category of effective MW-motives fits into the following diagram of adjoint functors (where $R$ is a ring)
\begin{equation}\label{eq:unstable}
\begin{split}
\xymatrix@C=30pt@R=24pt{
\DAe\ar@<+2pt>^{\derL \tilde \gamma^*}[r]
 & \DMteR\ar@<+2pt>^{\derL \pi^*}[r]
     \ar@<+2pt>^{\tilde \gamma_{*}}[l]
 & \DMe
     \ar@<+2pt>^{\pi_{*}}[l]
}
\end{split}
\end{equation}
where the left-hand category is the effective $\A^1$-derived category (whose construction is for instance recalled in \cite[\S 1]{Deglise17}).

More precisely, each category is the homotopy category of a proper cellular model category and the functors, which are defined at the level of the underlying closed model categories, are part of a Quillen adjunction. Moreover, each model structure is symmetric monoidal, the respective tensor products admit a total left derived functor and the corresponding internal homs admit a total right derived functor. The left adjoints are all monoidal and send representable objects to the corresponding representable object, while the functors from right to left are conservative. The corresponding diagram for stable categories reads as
\begin{equation}\label{eq:stable}
\begin{split}
\xymatrix@C=30pt@R=24pt{
\DA\ar@<+2pt>^{\derL \tilde \gamma^*}[r]
 & \DMtR\ar@<+2pt>^{\derL \pi^*}[r]
     \ar@<+2pt>^{\tilde \gamma_{*}}[l]
 & \DM
     \ar@<+2pt>^{\pi_{*}}[l] 
}
\end{split}
\end{equation}
and enjoys the same properties as in the unstable case.


\section{Rational contractibility}\label{sec:rational}

Recall the following definition from \cite[\S 2]{Suslin03}. For any presheaf $F$ of abelian groups, let $\tilde{C}_1F$ be the presheaf defined by
\[
\tilde{C}_1F(X)=\colim_{X\times\{0,1\}\subset U\subset X\times \A^1} F(U),
\]
where $U$ ranges over open subschemes of $X \times \A^1$ containing $X \times \{0,1\}$.
Observe that the restriction of $\tilde{C}_1F(X)$ to both $X\times \{0\}$ and $X\times\{1\}$ make sense, i.e. that we have morphisms of presheaves $i_0^*:\tilde{C}_1F\to F$ and $i_1^*: \tilde{C}_1F\to F$.

\begin{defin}
A presheaf $F$ is called rationally contractible if there exists a morphism of presheaves $s:F\to \tilde{C}_1F$ such that $i_0^*s=0$ and $i_1^*s=\id_F$.
\end{defin}

We note the following stability property.

\begin{lem} \label{lemm:rat-contractible-pullback}
Let $K/k$ be a field extension and write $p: Spec(K) \to Spec(k)$ for the associated morphism of schemes. Then $p^* \tilde{C}_1F \simeq \tilde{C}_1 p^*F$. In particular, $p^*F$ is rationally contractible if $F$ is.
\end{lem}
\begin{proof}
Since $k$ is perfect, $p$ is essentially smooth and so for $X \in Sm_K$ there exists a cofiltered diagram with affine transition maps $\{X_i\} \in Sm_k$ with $X = \lim_i X_i$. Then for any sheaf $G$ on $Sm_k$ we have $(p^*G)(X) = \colim_i G(X_i)$. Now, note that $X\times \A^1= \lim_i (X_i\times_k \A^1)$ and \cite[Corollaire 8.2.11]{EGAIV3} shows that any open subset in $X \times \A^1$ containing $X \times \{0,1\}$ is pulled back from an open subset of $X_i \times \A^1$ containing $X_i \times \{0, 1\}$ for some $i$. The result follows.
\end{proof}

The main property of rationally contractible presheaves is the following result which we will use later. 

\begin{prop}[Suslin] \label{prop:suslin}
Let $F$ be a rationally contractible presheaf of abelian groups on $Sm_k$. Then
$(\Cstar F)(\hat\Delta^\bullet_K) \simeq 0$, for any field $K/k$.
\end{prop}
\begin{proof}
Combine \cite[Lemma 2.4 and Proposition 2.5]{Suslin03}, and use Lemma \ref{lemm:rat-contractible-pullback}.
\end{proof}

Examples of rationally contractible presheaves are given in \cite[Proposition 2.2]{Suslin03}, and we give here a new example that will be very useful in the proof of our main result. 

\begin{prop}\label{prop:ratcontractible}
Let $X$ be a smooth connected scheme over $k$ and $x_0\in X$ be a rational $k$-point of $X$. Assume that there exists an open subscheme $W\subset X\times \A^1$ containing $(X\times \{0,1\})\cup (x_0\times \A^1)$ and a morphism of schemes $f:W\to X$ such that $f_{\vert_{X\times 0}}=x_0$, $f_{\vert_{X\times 1}}=\id_X$ and $f_{\vert_{x_0\times \A^1}}=x_0$. Then the presheaf $\tilde{\mathrm{c}}(X)/\tilde{\mathrm{c}}(x_0)$ is rationally contractible.
\end{prop}

\begin{proof}
We follow closely Suslin's proof in \cite[Proposition 2.2]{Suslin03}. Let $Y$ be a smooth connected scheme and let $\alpha\in \cor k(Y,X)$. There exists then an admissible subset $Z\subset Y\times X$ (i.e. $Z$ endowed with its reduced structure is finite and surjective over $X$) such that 
\[
\alpha\in \chst {d_X}{Z}{Y\times X}{\omega_X}.
\]
where $\omega_X$ is the pull-back along the projection $Y\times X\to X$ of the canonical sheaf of $X$.
On the other hand, the class of $\tilde\gamma(\id_{\A^1})$ is given by the class of the MW-correspondence $\Delta_*(\langle 1\rangle)$ where 
\[
\Delta_*:\ch 0{\A^1}\to \chst 1{\Delta}{\A^1\times \A^1}{\omega_{\A^1}}
\]
is the push-forward along the diagonal $\Delta:\A^1\to \A^1\times \A^1$, and $\Delta=\Delta(\A^1)$. Considering the Cartesian square
\[
\xymatrix{Y\times X\times \A^1\times \A^1\ar[r]^-{p_2}\ar[d]_-{p_1} & \A^1\times \A^1\ar[d] \\
Y\times X\ar[r] & \spec k}
\]
we may form the exterior product $p_1^*\alpha\cdot p_2^*\Delta_*(\langle 1\rangle)$ and its image under the push-forward along $\sigma:Y\times X\times \A^1\times\A^1\to Y\times\A^1\times X\times \A^1$ represents the MW-correspondence $\alpha\times \id_{\A^1}$ defined in \cite[\S 4.4]{Calmes14b}. Using this explicit description, we find a cycle
\[
\alpha\times \id_{\A^1}\in \chst {d_X+1}{Z\times \Delta}{Y\times \A^1\times X\times \A^1}{\omega_{X\times \A^1}}
\]
where $Z\times \Delta$ is the product of $Z$ and $\Delta$. Now, we may consider the closed subset $T:=(X\times \A^1)\setminus W\subset X\times \A^1$. It is readily verified that $T^\prime:=(Z\times \Delta)\cap (Y\times \A^1\times T)$ is finite over $Y\times \A^1$. Thus $p_{Y\times \A^1}(T^\prime)\subset Y\times \A^1$ is closed and we can consider its open complement $U$ in $(Y\times \A^1)$. It follows from \cite[proof of Proposition 2.2]{Suslin03} that $Y\times \{0,1\}\subset U$. By construction, we see that $\left(U\times (X\times \A^1)\right)\cap (Z\times \Delta)\subset U\times W$ and is finite over $U$. Restricting $\alpha\times \id_{\A^1}$ to $U\times W$, we find 
\[
i^*(\alpha\times \id_{\A^1})\in \chst {d_X+1}{(Z\times \Delta)\cap (U\times W) }{U\times W}{i^*\omega_{X\times \A^1}}
\]
where $i:U\times W\to Y\times \A^1\times X\times \A^1$ is the inclusion. Now, we see that we have a canonical isomorphism $i^*\omega_{X\times \A^1}\simeq \omega_W$ and it follows that we can see $i^*\beta$ as a finite $\mathrm{MW}$-correspondence between $U$ and $W$. Composing with $f:W\to X$, we get a finite MW-correspondence $f\circ s(\alpha):U\to X$, i.e. an element of $\cor k(U,X)=\tilde{\mathrm{c}}(X)(U)$ with $Y\times \{0,1\}\subset U\subset Y\times \A^1$. Using now the canonical morphism $\tilde{\mathrm{c}}(X)(U)\to  \tilde{C}_1(\tilde{\mathrm{c}}(X))(Y)$, we obtain an element denoted by $s(\alpha)$. It is readily checked that this construction is (contravariantly) functorial in $Y$ and thus that we obtain a morphism of presheaves
\[
s:\tilde{\mathrm{c}}(X)\to \tilde{C}_1(\tilde{\mathrm{c}}(X)).
\]
We check as in \cite[Proposition 2.2]{Suslin03} that this morphism induces a morphism
\[
s:\tilde{\mathrm{c}}(X)/\tilde{\mathrm{c}}(x_0)\to \tilde{C}_1(\tilde{\mathrm{c}}(X)/\tilde{\mathrm{c}}(x_0)).
\]
with the prescribed properties.
\end{proof}

\begin{coro}\label{cor:gmcase}
For any $n\geq 1$, the presheaf $\tilde{\mathrm{c}}(\gm^{\times n})/\tilde{\mathrm{c}}(1,\ldots,1)$ is rationally contractible.
\end{coro}

\begin{proof}
Let $t_1,\ldots,t_n$ be the coordinates of $\gm^{\times n}$ and $u$ be the coordinate of $\A^1$. We consider the open subscheme $W\subset \gm^{\times n}\times \A^1$ defined by $ut_i+(1-u)\neq 0$. It is straightforward to check that $\gm^{\times n}\times \{0,1\}\subset W$ and that $(1,\ldots,1)\times \A^1\subset W$. We then define 
\[
f:W\to \gm^{\times n}
\] 
by $f(t_1,\ldots,t_n,u)=u(t_1,\ldots,t_n)+(1-u)(1,\ldots,1)$ and check that it fulfills the hypothesis of Proposition \ref{prop:ratcontractible}.
\end{proof}

We would like to deduce from this result that Proposition \ref{prop:suslin} also holds for the sheaf $\tilde{\Z}(\gm^{\times n})/\tilde{\Z}(1,\ldots,1)$ associated to the presheaf $\tilde{\mathrm{c}}(\gm^{\times n})/\tilde{\mathrm{c}}(1,\ldots,1)$, or more precisely that it holds for its direct summand $\tilde{\Z}(\gm^{\wedge n}):=\tilde{\Z}\{n\}$ for $n\geq 1$. This requires some comparison results between the Suslin complex of a presheaf and the Suslin complex of its associated sheaf, which are the objects of the next section.


\section{Semi-local schemes}\label{sec:semilocal}

In this section, a \emph{semi-local scheme} will be a localization of a smooth integral scheme $X$ at finitely many points.

Our aim in this section is to extend \cite[Corollary 4.0.4]{Fasel16b} to the case of semi-local schemes. Let us first recall a result of H. Kolderup (\cite[Theorem 3.1]{Kolderup17}).

\begin{thm}\label{thm:excision}
Let $X$ be a smooth $k$-scheme and let $x\in X$ be a closed point. Let $U=\spec {\OO_{X,x}}$ and let $\mathrm{can}:U\to X$ be the canonical inclusion. Let $i:Z\to X$ be a closed subscheme with $x\in Z$ and let $j:X\setminus Z\to X$ be the open complement. Then there exists a finite $\mathrm{MW}$-correspondence $\Phi\in \cor k(U,X\setminus Z)$ such that the following diagram 
\[
\xymatrix{ & X\setminus Z\ar[d]^-j \\
U\ar[r]_{\mathrm{can}}\ar[ru]^-{\Phi} & X}
\]
commutes up to homotopy.
\end{thm}

We note that this result uses a proposition of Panin-Stavrova-Vavilov (\cite[Proposition 1]{Panin09}) which is in fact true for the localization of a smooth scheme at finitely many closed points and that the proof of Theorem \ref{thm:excision} goes through in this setting. This allows us to prove the following corollary. We thank M. Hoyois for pointing out the reduction to closed points used in the proof.

\begin{coro}
Let $X$ be a smooth scheme and let $x_1,\ldots,x_n\in X$ be finitely many points. Let $U=\spec{\OO_{X,{x_1,\ldots,x_n}}}$ and let $\mathrm{can}:U\to X$ be the inclusion. Let $i:Z\to X$ be a closed subscheme containing $x_1,\ldots,x_n$ and let $j:X\setminus Z\to X$ be the open complement. Then, there exists a finite $\mathrm{MW}$-correspondence $\Phi\in \cor k(U,X\setminus Z)$ such that the following diagram 
\[
\xymatrix{ & X\setminus Z\ar[d]^-j \\
U\ar[r]_{\mathrm{can}}\ar[ru]^-{\Phi} & X}
\]
commutes up to homotopy.
\end{coro}

\begin{proof}
Let $v_1,\ldots,v_n$ be (not necessarily distinct) closed specializations of $x_1,\ldots,x_n$ and let $V$ be the semi-localization of $X$ at these points. We have a sequence of inclusions $U\stackrel{\iota}\to V\stackrel{\mathrm{can}}\to X$. As $Z$ is closed, we see that $v_1,\ldots,v_n$ are also in $Z$ and we may apply the previous theorem to get a finite MW-correspondence $\Phi^\prime$ and a homotopy commutative diagram
\[
\xymatrix{ & X\setminus Z\ar[d]^-j \\
V\ar[r]_{\mathrm{can}}\ar[ru]^-{\Phi^\prime} & X.}
\]
Composing with the map $U\stackrel{\iota}\to V$, we get the result with $\Phi=\Phi^\prime\circ \iota$.
\end{proof}

We deduce the next result from the above, following \cite[Corollary 11.2]{Kolderup17}.

\begin{coro}\label{cor:restriction}
Let $F$ be a homotopy invariant presheaf with $\mathrm{MW}$-transfers. Let $Y$ be a semi-local scheme. Then the restriction homomorphism $F(Y)\to F(k(Y))$ is injective.
\end{coro}

\begin{proof}
Let $Y$ be the semi-localization of the smooth integral $k$-scheme $X$ at the points $x_1, \dots, x_n$.  By definition, we have $F(Y)=\colim_{x_1,\ldots,x_n\in V}F(V)$, whereas $F(k(Y)) = F(k(X))=\colim_{W\neq \emptyset}F(W)$. Here $V,W$ are open subschemes of $X$. Let then $s\in \colim_{x_1,\ldots,x_n\in V}F(V)$ mapping to $0$ in $F(k(X))$. There exists $V$ containing $x_1\ldots,x_n$ and $t\in F(V)$ such that $s$ is the image of $t$ under the canonical homomorphism, and there exists $W\neq \emptyset$ such that $t_{\vert_{W\cap V}}=0$. Shrinking $W$ if necessary, we may assume that $x_1,\ldots,x_n\not\in W$. We can now use Theorem \ref{thm:excision} with $X=V$, $Y=U$ and $Z=V\setminus (V\cap W)$. Since $F$ is homotopy invariant, we then find a commutative diagram
\[
\xymatrix{ & F(V\cap W)\ar[ld]_-{\Phi^*} \\
F(Y) & \ar[l]^-{\mathrm{can}^*}F(V)\ar[u]_-{j^*}}
\]
showing that $s=0$.
\end{proof}

\begin{coro}
Let $F\to G$ be a morphism of homotopy invariant MW-presheaves such that for any finitely generated field extension $L/k$ the induced morphism $F(L)\to G(L)$ is an isomorphism. Then the homomorphism $F(X)\to G(X)$ is an isomorphism for any semi-local scheme $X$.
\end{coro}

\begin{proof}
As the category of MW-presheaves is abelian, we can consider both the kernel $K$ and the cokernel $C$ of $F\to G$. An easy diagram chase shows that $C$ and $K$ are homotopy invariant and our assumption implies that $C(L)=0=K(L)$ for any finitely generated field extension $L/k$. By Corollary \ref{cor:restriction}, it follows that $C(X)=0=K(X)$, proving the claim.
\end{proof}

\begin{coro}\label{cor:equality}
Let $F$ be a homotopy invariant MW-presheaf. Let respectively $F_{\Zar}$ be the associated Zariski sheaf and $F_{\Nis}$ be the associated Nisnevich sheaf. Then the canonical sequence of morphisms of presheaves
\[
F\to F_{\Zar} \to F_{\Nis}
\]
induces isomorphisms $F(X)\simeq F_{\Zar}(X)\simeq F_{\Nis}(X)$ for any semi-local scheme $X$.
\end{coro}

\begin{proof}
First note that $F_{\Nis}$ is indeed an MW-sheaf by \cite[Proposition 1.2.11]{Deglise16}. Moreover, the associated Zariski sheaf $F_{\Zar}$ coincides with $F_{\Nis}$ and they are both homotopy invariant by \cite[Theorem 3.2.9]{Deglise16}. To conclude, we observe that the sequence $F\to F_{\Zar} \to F_{\Nis}$ induces isomorphisms when evaluated at finitely generated field extensions and we can use the previous corollary to obtain the result.
\end{proof}

We now pass to the identification of the higher cohomology presheaves of the sheaf associated to a homotopy invariant MW-presheaf $F$.

\begin{lem}\label{lem:superior}
Let $F$ be a homotopy invariant MW-presheaf. Then $\H^n_{\Zar}(X,F_{\Zar})=\H^n_{\Nis}(X,F_{\Nis})=0$ for any semi-local scheme $X$ and any $n>0$.
\end{lem}

\begin{proof}
Using \cite[Theorem 3.2.9]{Deglise16}, it suffices to prove the result for $F_{\Nis}$. Now, the presheaf $U\mapsto \H^n_{\Nis}(U,F_{\Nis})$ is an MW-presheaf (as the category of MW-sheaves has enough injectives by \cite[Proposition 1.2.11]{Deglise16} and \cite[Th\'eor\`eme 1.10.1]{Grothendieck57}) which is homotopy invariant by \cite[Theorem 3.2.9]{Deglise16} again. As any field has Nisnevich cohomological dimension $0$, we find $\H^n_{\Nis}(L,F_{\Nis})=0$ for any finitely generated field extension $L/k$. We conclude using Corollary \ref{cor:restriction}.
\end{proof}

Recall that $\DMteZ$ is the homotopy category of a certain model category. This model category is obtained as a localization of a model structure on the category $C(\sh{k}{\Nis})$ of unbounded chain complexes of MW-sheaves. We call a fibrant replacement functor for this localized model structure the \emph{$\mathrm{MW}_{\mot}$-localization} functor, and denote it $\mathrm{L}_{\mot}$. If $K$ is a complex of MW-presheaves, then we can take the associated complex of Nisnevich MW-sheaves $a_\Nis K$. We write $\mathrm{L}_\Nis K$ for a fibrant replacement of $a_\Nis K$ in the usual (i.e. non-$\A^1$-localized) model structure on $C(\sh{k}{\Nis})$ (\cite[\S 3.1]{Deglise16}).

We will need the following slight strengthening of \cite[Corollary 3.2.14]{Deglise16}.
\begin{lem} \label{lem:motivic-localization}
Let $F$ be an MW-presheaf. Then the motivic localization (of $F_{Nis}$) is given by $\mathrm{L}_{\Nis} \Cstar F$.
\end{lem}
\begin{proof}
Throughout the proof we abbreviate $\Delta^\bullet := \Delta^\bullet_k$.
We claim that $F_\Nis$ and $a_\Nis \Cstar F$ are $\A^1$-equivalent. To see this, let $\CstarS F$ denote the complex constructed like $\Cstar F$, but with the constant cosimplicial object $*$ in place of $\Delta^\bullet$. In other words $\CstarS F = F \xleftarrow{0} F \xleftarrow{1} F \xleftarrow{0} \dots$. The projection $\Delta^\bullet \to *$ induces $\alpha: \CstarS F \to \Cstar F$. Since $\CstarS F$ is chain homotopy equivalent to $F$, it will suffice to show that $a_\Nis \alpha: a_\Nis \CstarS F \to a_\Nis \Cstar F$ is an $\A^1$-equivalence. For this, it is enough to prove that $a_\Nis \alpha$ is a levelwise $\A^1$-equivalence (because $\A^1$-equivalences are closed under filtered colimits), for which in turn it is enough to prove that $\alpha$ is a levelwise $\A^1$-homotopy equivalence. This is clear, since $\alpha_n$ is $F \to F^{\Delta^n}$, and $\Delta^n$ is $\A^1$-contractible. This proves the claim.
It thus remains to show that $a_\Nis \Cstar F$ is $\A^1$-local. This follows from \cite[Corollary 3.2.11]{Deglise16}.
\end{proof}

\begin{coro} \label{coro:compute-semilocal-suslin}
Let $F$ be a MW-presheaf and let $\Cstar(F)$ be its associated Suslin complex. For any $n\in\Z$, let $\H^n(\Cstar(F))$ be the $n$-th cohomology presheaf of $\Cstar(F)$. Then for any semi-local scheme $X$ over $k$, we have canonical isomorphisms
\[
\H^n(\Cstar(F))(X)\to \mathbb{H}^n_{\Nis}(X,\mathrm{L}_{\mot} F_\Nis).
\]
\end{coro}

\begin{proof}
By Lemma \ref{lem:motivic-localization}, we have $\mathrm{L}_{\mot} F_\Nis \simeq \mathrm{L}_\Nis \Cstar F$.
Observe first that the cohomology presheaves $\H^n(\Cstar(F))$ are homotopy invariant and have MW-transfers. Denote by $h^n_{\Nis}$ the associated Nisnevich sheaves (which are homotopy invariant MW-sheaves by \cite[Theorem 3.2.9]{Deglise16}). Considering the hypercohomology spectral sequence, we see that it suffices to prove that $\H^n_{\Nis}(\Cstar(F))(X)=\H^0(X,h^n_{\Nis})$ and that $\H_{\Nis}^i(X,h^n_{\Nis})=0$ for $i>0$. The first claim follows from Corollary \ref{cor:equality}, while the second one follows from Lemma \ref{lem:superior}.
\end{proof}

\begin{rem}
Using the fact that the Zariski sheaf $h^n_{\Zar}$ associated to $\H^n(\Cstar(F))$ coincides with $h^n_{\Nis}$ (\cite[Theorem 3.2.9]{Deglise16}), the same arguments as above give a canonical isomorphism
\[
\H^n(\Cstar(F))(X)\to \mathbb{H}^n_{\Zar}(X,\Cstar(F_{\Zar})).
\]
\end{rem}

Finally, we are in position to prove the result we need. In the statement, the complexes are the total complexes associated to the relevant bicomplexes of abelian groups.

\begin{coro} \label{coro:compute-semilocal-suslin-2}
Let $F$ be a MW-presheaf and $K/k$ be a finitely generated field extension. The canonical map
\[ \Cstar(F)(\hat\Delta^\bullet_K) \to (\mathrm{L}_{\mot}F_{\Nis})(\hat\Delta^\bullet_K) \]
is a weak equivalence of complexes of abelian groups.
\end{coro}

\begin{proof}
We have strongly convergent spectral sequences
\[ \H^p(\Cstar(F)(\hat\Delta^q_K)) \Rightarrow \H^{p+q}(\Cstar(F)(\hat\Delta^\bullet_K)) \]
and
\[ 
\H^p((\mathrm{L}_{\mot} F_\Nis)(\hat\Delta^q_K)) \Rightarrow \H^{p+q}((\mathrm{L}_{\mot} F_\Nis)(\hat\Delta^\bullet_K)). 
\]
Since $\mathrm{L}_{\mot} F_\Nis$ is Nisnevich-local, we have $\H^p((\mathrm{L}_{\mot} F_\Nis)(\hat\Delta^q_K)) = \mathbb{H}^p_{\Nis}(\hat\Delta^q_K, \mathrm{L}_{\mot} F_\Nis)$.
Thus the claim follows from Corollary \ref{coro:compute-semilocal-suslin} and spectral sequences comparison. Here we use that $\hat\Delta^q_K$ is semilocal: if $K = k(U)$ for some smooth irreducible scheme with generic point $\eta$, then $\hat\Delta^q_K$ is the semilocalization of $\Delta^q \times U$ in the points  $(v_i, \eta)$.
\end{proof}

\begin{thm}\label{thm:rational}
For any $n\geq 1$ and $K/k$ finitely generated, we have 
\[
\mathrm{L}_{\mot}(\tilde{\Z}(n))(\hat\Delta^\bullet_K)\simeq 0.
\]
\end{thm}

\begin{proof}
Since $\tilde{\Z}(n)[n]$ is motivically equivalent to $\tilde{\mathrm{c}}(\gm^{\wedge n})$, and the latter is a direct factor of $\tilde{\mathrm{c}}(\gm^{\times n})/\tilde{\mathrm{c}}(1,\dots,1)$, by Corollary \ref{coro:compute-semilocal-suslin-2} it suffices to show that $\Cstar(\tilde{\mathrm{c}}(\gm^{\times n})/\tilde{\mathrm{c}}(1,\dots,1))(\hat\Delta^\bullet_K)\simeq 0$. This follows from Corollary \ref{cor:gmcase} and Proposition \ref{prop:suslin}.
\end{proof}


\section{A General Criterion}
\label{sec:general-criterion}

In this section we study when the motivic spectrum representing a generalized
cohomology theory of algebraic varieties is effective. We first recall a few facts about the slice filtration of \cite{voevodsky-slice-filtration}.

Let $\SHS$ be the motivic homotopy category of $S^1$-spectra and let $\SH$ be the stable motivic homotopy category. We have an adjunction
\[
\Sigma^\infty_T: \SHS \leftrightarrows \SH:\Omega^{\infty}_T
\]
and we write $\SH^\eff$ for the localising subcategory (in the sense of \cite[3.2.6]{Neeman}) of $\SH$
generated by the image of $\Sigma^\infty_{T}$. The inclusion $i_0:\SH^\eff\to \SH$ has a right adjoint $r_0:\SH \to \SH^\eff$ and we obtain a functor $f_0=i_0r_0:\SH \to \SH$ called the effective cover functor. More generally, we may consider the localising subcategories $ \SHS(d)$ and $\SH^\eff(d)$ of respectively $\SHS$ and $\SH$ generated by the images of $X\wedge T^d$ for $X$ smooth and $d\in\N$. We obtain a commutative diagram of functors
\[
\xymatrix{\SHS(d)\ar[r]^-{\Sigma_T^\infty}\ar[d]_-{i_d} & \SH^{\eff}(d)\ar[d]^-{i_d} \\
\SHS\ar[r]_-{\Sigma_T^\infty} & \SH}
\]
Both of the inclusions $i_d:\SHS(d)\to \SHS$ and $i_d:\SH^\eff(d)\to \SH$ admit right adjoints $r_d$ and we set $f_d=i_dr_d$ (on both categories). We obtain a sequence of endofunctors 
\[
\ldots\to f_d\to f_{d-1}\to\ldots \to f_1\to f_0
\]
and we define $s_0$, the \emph{zeroth slice functor}, as the cofiber of $f_1\to f_0$. More generally, we let $s_d$ be the cofiber of $f_{d+1}\to f_d$.

The following result is due to M. Levine (\cite[Theorems 9.0.3 and 7.1.1]{levine2008homotopy}).

\begin{lem}\label{lem:omegas0}
The following diagram of functors
\[
\xymatrix{\SH\ar[r]^-{\Omega_T^{\infty}}\ar[d]_-{s_0} & \SHS\ar[d]^-{s_0}  \\
\SH\ar[r]_-{\Omega_T^\infty} & \SHS }
\]
is commutative.
\end{lem}

One essential difference between $\SHS$ and $\SH$ is that in the latter case, the above sequence of functors extends to a sequence of endofunctors
\[
\ldots\to f_d\to f_{d-1}\to\ldots \to f_1\to f_0\to f_{-1}\to \ldots\to f_{-n}\to\ldots
\]
Let us recall the following well-known lemma for the sake of completeness.

\begin{lem} \label{lemm:eff}
Let $E \in \SH$. Then $\hocolim_{n \to \infty} f_{-n} E \to E$ is an equivalence.
\end{lem}

\begin{proof}
It suffices to show that for any $X \in \mathrm{Sm}_k$ and any $i, j \in \Z$ we get 
\[
\mathrm{Hom}_{\SH}(\Sigma^\infty X_+ [i] \wedge \gm^{\wedge j}, E) = \mathrm{Hom}_{\SH}(\Sigma^\infty X_+ [i] \wedge \gm^{\wedge j}, \hocolim f_{-n} E).
\] 
Since $\Sigma^\infty X_+ [i] \wedge \gm^{\wedge j}$ is compact, the right hand side is equal to $\colim_n \mathrm{Hom}_{\SH}(\Sigma^\infty X_+ [i] \wedge \gm^{\wedge j}, f_{-n} E)$. For $j > -n$, we have $\Sigma^\infty X_+ [i] \wedge \gm^{\wedge j} \in \SH^\eff(-n)$ and hence 
\[
\mathrm{Hom}_{\SH}(\Sigma^\infty X_+ [i] \wedge \gm^{\wedge j}, f_{-n} E) = \mathrm{Hom}_{\SH}(\Sigma^\infty X_+ [i] \wedge \gm^{\wedge j}, E).
\] 
The result follows.
\end{proof}

We now make use of the \emph{spectral enrichment} of $\SH$. To explain it, consider $E \in \SH$. This yields a presheaf $rE \in PSh(\mathrm{Sm}_k)$ given by $(rE)(U) = \mathrm{Hom}_{\SH}(\Sigma^\infty_T U_+, E)$. Write $\Spt(\mathrm{Sm}_k)$ for the homotopy category of \emph{spectral presheaves} (\cite[\S 1.4]{Levine06}). Then there exists a functor $R: \SH \to \Spt(\mathrm{Sm}_k)$ such that $rE = \pi_0 RE$. Indeed $R$ is constructed as the following composite
\[ \SH \xrightarrow{\Omega^\infty_T} \SHS \xrightarrow{R_0} \Spt(\mathrm{Sm}_k), \]
where $R_0$ is the (fully faithful) right adjoint of the localization functor. Now note that if $P \in \Spt(\mathrm{Sm}_k)$ is a spectral presheaf and $F/k$ is a finitely generated field extension, then we can make sense of the expression $P(\hat\Delta^\bullet_F) \in \mathrm{SH}$: it is obtained by choosing a bifibrant model of $P$ as a presheaf of spectra, and then taking the geometric realization of the induced simplicial diagram \cite[1.5]{Levine06}. If $E \in \SH$, then we abbreviate $(RE)(\hat\Delta^\bullet)$ to $E(\hat\Delta^\bullet)$. Similarly if $E \in \SHS$, then we abbreviate $(R_0E)(\hat\Delta^\bullet)$ to $E(\hat\Delta^\bullet)$

\begin{lem} \label{lemm:s0-vanishing}
Let $E \in \SH$, where $k$ is a perfect field. Then $s_0(E) \simeq 0$ if and only if for all finitely generated fields $F/k$ we have $E(\hat\Delta^\bullet_F) \simeq 0$.
\end{lem}

\begin{proof}
By definition, we have an exact triangle 
\[
f_1E\to f_0E\to s_0(E)\to f_1E[1]
\]
and it follows that $s_0(E)\in \SH^{\eff}$. On the other hand, the adjunction between the stable categories induces an adjunction 
\[
\Sigma^\infty_T: \SHS \leftrightarrows \SH^\eff: \Omega^\infty_T
\] 
and $\Omega^\infty_T$ is conservative on $\SH^\eff$ (its left adjoint has dense image). Thus $s_0(E) \simeq 0$ if and only if $\Omega_T^{\infty}s_0(E)\simeq 0$, and the latter condition is equivalent to $s_0\Omega_T^{\infty}E\simeq 0$ by Lemma \ref{lem:omegas0}. By definition, we have $(\Omega_T^{\infty}E)(\hat \Delta^\bullet_F) = E(\hat \Delta^\bullet_F)$, and
we are thus reduced to proving that for $E\in \SHS$, we have $s_0(E)\simeq 0$ if and only if $E(\hat\Delta^\bullet_F)=0$ for $F/k$ finitely generated.

Let then $E\in \SHS$. We can (and will) choose a fibrant model for $E$, which we denote by the same letter. Now $s_0(E)$ is given by the $E^{(0/1)}$ construction of Levine (\cite[Theorem 7.1.1]{levine2008homotopy}) and then $s_0(E) \simeq 0$ if and only if we have $E^{(0/1)}\simeq 0$. Since strictly homotopy invariant sheaves are unramified (\cite[Example 2.3]{Morel08}), $E^{(0/1)}\simeq 0$ if and only if $E^{(0/1)}(F)\simeq 0$ for any finitely generated field extension $F/k$. Since $E^{(0/1)}(F) \simeq E(\hat \Delta^\bullet_F)$ (this argument is used for example in \cite[proof of Lemma 5.2.1]{levine2008homotopy}), this concludes the proof.
\end{proof}

\begin{thm} \label{thm:eff-crit}
Let $E \in \SH$, where $k$ is a perfect field. Then $E \in \SH^\eff$ if and only if for all $n \ge 1$ and all finitely generated fields $F/k$, we have $(E \wedge \gm^{\wedge n})(\hat \Delta^\bullet_F) \simeq 0$.
\end{thm}
\begin{proof}
By Lemma \ref{lemm:s0-vanishing}, we know that the condition is equivalent to $s_0(E \wedge \gm^{\wedge n}) \simeq 0$. This is clearly necessary for $E \in \SH^\eff$ and we are left to prove sufficiency.

Note that $s_0(E \wedge \gm^{\wedge n}) \simeq s_{-n}(E) \wedge \gm^{\wedge n}$. Thus our condition is equivalent to $s_{-n}(E) \simeq 0$ for all $n \ge 1$, or equivalently $f_0(E) \simeq f_{-n}(E)$ for all $n \ge 0$. Consequently we get $f_0(E) \simeq \hocolim_n f_{-n}(E)$. But this homotopy colimit is equivalent to $E$, by Lemma \ref{lemm:eff}. This concludes the proof.
\end{proof}

\begin{coro} \label{coro:main}
Let $\mathcal{D}$ be a symmetric monoidal category and let 
\[
M:\SH  \leftrightarrows \mathcal{D}:U
\] 
be a pair of adjoint functors such that $M$ is symmetric monoidal.
Then $U(\1_{\mathcal{D}}) \in \SH^\eff$ if and
only if $U(M \gm^{\wedge n})(\hat \Delta^\bullet_F) \simeq 0$ for all $F/k$
finitely generated and all $n \ge 1$.
\end{coro}

\begin{proof}
Let $E = U(\1_{\mathcal{D}})$. Note that by Lemma \ref{lemm:inv} below, we have $U(M(\gm^{\wedge n})) \simeq E \wedge \gm^{\wedge n}$. Thus the result reduces to Proposition \ref{thm:eff-crit}.
\end{proof}

For the convenience of the reader, we include a proof of the following well-known result.

\begin{lem} \label{lemm:inv}
Let $M: \mathcal{C} \leftrightarrows \mathcal{D}: U$ be an adjunction of
symmetric monoidal categories, with $M$ symmetric monoidal. Then for any
rigid (e.g. invertible) object $G \in \mathcal{C}$ and any $E \in \mathcal{D}$, there is a
canonical isomorphism $U(E \wedge MG) \simeq U(E) \wedge G$.
\end{lem}

\begin{proof}
Let $DG$ be the dual object of $G$. As $M$ is symmetric monoidal, we see that $MG$ also admits a dual object, namely $M(DG)$.
For any object $F \in \mathcal{C}$, we get $\mathrm{Hom}_{\mathcal C}(F, U(E \wedge MG)) = \mathrm{Hom}_{\mathcal D}(MF, E
\wedge MG) = \mathrm{Hom}_{\mathcal D}(MF\wedge M(DG), E) = \mathrm{Hom}_{\mathcal D}(M(F \wedge DG), E) = \mathrm{Hom}_{\mathcal C}(F \wedge DG, UE) = \mathrm{Hom}_{\mathcal C}(F, UE \wedge G)$. Thus we conclude by the Yoneda lemma.
\end{proof}

We can simplify this criterion in a special case.

\begin{coro} \label{corr:main-simplified}
Consider the following diagram of functors
\[
\begin{split}
\xymatrix@C=30pt@R=24pt{
\SHS\ar@<+2pt>^-{M_0}[r]\ar[d]_-{\Sigma^\infty_T}
 & \mathcal D_0\ar[d]^-{L}
     \ar@<+2pt>^-{U_0}[l] \\
\SH\ar@<+2pt>^-{M}[r]
 & \mathcal D\ar@<+2pt>^-{U}[l]
}
\end{split}
\]
where the rows are adjunctions, $M_0, M$ and $L$ are symmetric monoidal and $LM_0 \simeq M\Sigma^\infty_T$.
Suppose furthermore that $L$ is fully faithful and has a right adjoint $R$.

Then $U(\1_{\mathcal D}) \in \SH^\eff$ if and only if $U_0(M_0 \gm^{\wedge n})(\hat \Delta^\bullet_F) \simeq 0$ for $F$ as in Corollary \ref{coro:main}.
\end{coro}

\begin{proof}
First, observe that there is an isomorphism $\Omega^\infty_T U \simeq U_0 R$ since $LM_0 \simeq M\Sigma^\infty_T$. Moreover, $R L \simeq \id$ since $L$ is assumed to be fully faithful. For any $E \in \mathcal{D}_0$, we then get $\Omega^\infty_T ULE \simeq U_0 R L E \simeq U_0 E$. Next,
\[
(ULE)(\hat \Delta^\bullet_F) = (\Omega^\infty_T ULE)(\hat \Delta^\bullet_F) \simeq (U_0 E)(\hat \Delta^\bullet_F)
\] 
where the first equality is by definition.

By Corollary \ref{coro:main}, we have $U(\1_{\mathcal D}) \in \SH^\eff$ if and only if $U(M \gm^{\wedge n})(\hat \Delta^\bullet_F) \simeq 0$ for $F$ as stated. Note that $M \gm^{\wedge n} \simeq LM_0  \gm^{\wedge n}$ by assumption. Hence by the first paragraph, we find that $U(M \gm^{\wedge n})(\hat \Delta^\bullet_F) \simeq (U_0 M_0 \gm^{\wedge n})(\hat \Delta^\bullet_F)$. This concludes the proof.
\end{proof}


\section{Application to MW-Motives}\label{sec:MWmotives}

In this section, we apply the result of the previous section to the category of MW-motives. We have a diagram of functors 
\[
\begin{split}
\xymatrix@C=30pt@R=24pt{
\SHS\ar@<+2pt>[r]^-N\ar[d]_-{\Sigma^\infty_T}
 & \mathrm{D}_{\A^1}^{\mathrm{eff}}(k)\ar@<+2pt>^{\derL \tilde \gamma^*}[r]\ar[d]_-{\Sigma^\infty_T}
     \ar@<+2pt>[l]^-K
 & \widetilde{\mathrm{DM}}^{\mathrm{eff}}(k)\ar[d]_-{\Sigma^\infty_T}
     \ar@<+2pt>[l]^-{\gamma_*} \\
\SH\ar@<+2pt>[r]^-N
 & \mathrm{D}_{\A^1}(k)\ar@<+2pt>^{\derL \tilde \gamma^*}[r]
		\ar@<+2pt>[l]^-K
 & \widetilde{\mathrm{DM}}(k)
		\ar@<+2pt>[l]^-{\gamma_*}
}
\end{split}
\]
where the vertical functors are given by $T$-stabilization, the adjunctions in the right-hand square are those discussed in Section \ref{sec:recollections}, and the adjunctions in the left-hand square are derived from the classical Dold-Kan correspondence (see \cite[5.2.25]{Cisinski12} for the unstable version, and \cite[5.3.35]{Cisinski12} for the $\pone$-stable version). Both $N$ and $\derL \tilde \gamma^*$ commute with $T$-stabilization, and the stabilization functor 
\[
\Sigma_T^\infty: \widetilde{\mathrm{DM}}^{\mathrm{eff}}(k)\to \widetilde{\mathrm{DM}}(k)
\] 
is fully faithful by \cite[Corollary 5.0.2]{Fasel16b}. It follows that the diagram
\[
\begin{split}
\xymatrix@C=30pt@R=24pt{
\SHS\ar@<+2pt>[r]^-{\derL \tilde \gamma^*N}\ar[d]_-{\Sigma^\infty_T}
 & \widetilde{\mathrm{DM}}^{\mathrm{eff}}(k)\ar[d]^-{\Sigma^\infty_T}
     \ar@<+2pt>[l]^-{K\gamma_*} \\
\SH\ar@<+2pt>[r]^-{\derL \tilde \gamma^*N}
 & \widetilde{\mathrm{DM}}(k)
		\ar@<+2pt>[l]^-{K\gamma_*}
}
\end{split}
\]
satisfies the assumptions of Corollary \ref{corr:main-simplified}. We can thus apply Theorem \ref{thm:rational} to obtain the following result, where $M:=\derL \tilde \gamma^*N$ and $U:=K\gamma_*$.

\begin{coro} \label{coro:effective}
In the stabilized adjunction $M: \SH \leftrightarrows  \widetilde{\mathrm{DM}}(k,\Z): U$,
we have $U(\1) \in \SH^\eff$.
\end{coro}
\begin{proof}
Having Theorem \ref{thm:rational} and Corollary \ref{corr:main-simplified} at hand, the only subtle point is to show the following: if $E \in \DMteZ$ has a fibrant model still denoted by $E$, then $K_s(E(\hat\Delta^\bullet_F)) \simeq (U_0 E)(\hat\Delta^\bullet_F)$, where $U_0 = K\gamma^*$. Here $K_s: D(Ab) \to \mathrm{SH}$ denotes the classical stable Dold-Kan correspondence. Essentially this requires us to know that $K_s$ preserves homotopy colimits (at least we need filtered homotopy colimits and geometric realizations). This is well-known. In fact since this is a stable functor, it preserves all homotopy colimits if and only if it preserves arbitrary sums, if and only if its left adjoint preserves the compact generator(s), which is clear.
\end{proof}

We are now in position to prove our main result. To this end, recall that the motivic spectrum of abstract generalized motivic cohomology $\H\tZ \in \SH$ was defined in \cite[\S 4]{bachmann-very-effective} as the effective cover of the homotopy module of Milnor-Witt $K$-theory. Equivalently, $\H\tZ$ is the effective cover of the homotopy module $\{\underline\pi_{n,n}(\mathbb{S})\}_n$, where $\mathbb{S}$ is the sphere spectrum.

\begin{thm} \label{thm:comparison}
Let $k$ be an infinite perfect field of exponential characteristic $e \ne 2$ and let 
\[
M: \SH \leftrightarrows  \widetilde{\mathrm{DM}}(k): U
\]
be the above adjunction. Then the spectrum $U(\1)$ representing MW-motivic cohomology with $\Z$-coefficients is canonically isomorphic to the spectrum $\H\tZ$ representing abstract generalized motivic cohomology with $\Z$-coefficients. In particular, $U(\1) \in \SH^\eff$.
\end{thm}
\begin{proof}
For an effective spectrum $E \in \SH^\eff$, let $\tau_{\le 0}^\eff E \in \SH^{\eff}_{\le 0}$ denote the truncation in the effective homotopy $t$-structure \cite[Proposition 4]{bachmann-very-effective}.

We note that for $X$ local, (1) $H^{n, 0}(X, \tilde{\Z}) = 0$ for $n \ne 0$ and (2) $\H^{0, 0}(X, \tilde{\Z})= \sKMW_0(X)$ The unit map $\1 \to U(\1)$ induces $\alpha: \H\tZ \simeq \tau_{\le 0}^\eff \1 \to \tau_{\le 0}^\eff U(\1) \simeq U(\1)$, where the first equivalence is by definition and the second since $U(\1) \in \SH^\eff_{\le 0}$, by (1). Now $\alpha$ is a map of objects in $\SH^{\eff,\heart}$ (again by (1)) and hence an equivalence if and only if it induces an isomorphism on $\underline{\pi}_{0,0}$. This follows from (2).
\end{proof}

Next, we would like to show that ordinary motivic cohomology is represented by an explicit (pre-)sheaf in $\DMt$. We start with the following lemma (see also \cite[Theorem 5.3]{Garkusha17} and \cite[Theorem 1.1]{Elmanto17}).

\begin{lem} \label{lemm:modules}
Under the assumptions of the theorem, the category $\DMtinv$ is equivalent
to the category of highly structured modules over $U(\1_{\DMtinv}) \simeq H\tilde{\Z}[1/e]$.
\end{lem}

\begin{proof}
Let $\mathcal{M}$ be this category of modules. By abstract nonsense \cite[Construction 5.23]{mathew2017nilpotence} there is an
induced adjunction 
\[
M' :\mathcal{M} \leftrightarrows \DMtinv: U'
\] 
which satisfies $U'M' (\1_{\mathcal M}) \simeq \1_{\mathcal M}$. Under our assumptions, the category
$\SH[1/e]$ is compact-rigidly generated \cite[Corollary B.2]{levine2013algebraic} and hence so are the categories $\mathcal{M}$ and 
$\DMtinv$. It follows that $M'$ and $U'$ are inverse equivalences, see e.g.
\cite[Lemma 22]{bachmann-hurewicz}.
\end{proof}

\begin{coro} \label{coro:modules-Z}
Under the same assumptions, the presheaf $\Z \in \DMt$ represents
ordinary motivic cohomology with $\Z$-coefficients.
\end{coro}
\begin{proof}
Let $H = f_0 U(\Z)$. Then $\underline{\pi}_{0,0}(H) = \Z$ whereas $\underline{\pi}_{n,0}(H) = 0$ for $n \ne 0$. Also $\underline{\pi}_{-1,-1}(H) = (\underline{\pi}_{0,0}(H))_{-1} = 0$ and consequently $f_1 H = 0$, $s_0 H \simeq H$. The unit map $\1 \to U(\1) \to U(\Z)$ induces $\1 \to H$ and hence $\H\Z \simeq s_0(\1) \to s_0(H) \simeq H$. This is an equivalence since it is a map between objects in $\SH^{\eff,\heart}$ inducing an isomorphism on $\underline{\pi}_{0,0}(\bullet)$. We have thus found a canonical map $\alpha: \H\Z \to f_0 U(\Z) \to U(\Z)$, which we need to show is an equivalence. We show separately that $\alpha[1/e]$ and $\alpha[1/2]$ are equivalences; since $e \ne 2$ this is enough.

We claim that $U(\Z)[1/e] \in \SH^\eff$. This will imply that $\alpha[1/e]$ is an equivalence.
For $X \in \mathrm{Sm}_k$ we have $UM(X)[1/e] = \Sigma^\infty X_+ \wedge U(\1)[1/e]$, by the
previous lemma. In particular $UM(X)[1/e] \in \SH^\eff$. It follows that for $E
\in \DMte$ we get $U(E) \in \SH^\eff$ (indeed $U$ commutes with filtered colimits, being right adjoint to a functor preserving compact generators). This applies in particular to $E = \Z[1/e]$.

Recall that if $E \in \SH$, then $E[1/2]$ canonically splits into two spectra, which we denote by $E^+$ and $E^-$. They are characterised by the fact that the motivic Hopf map $\eta$ is zero on $E^+$ and invertible on $E^-$ \cite[Lemma 39]{bachmann-real-etale}.
Now consider $U(\Z)[1/2]$. The action of $\sKMW$ on $\underline{\pi}_{0,0}(U\Z) = \Z$ is by definition via the canonical epimorphism $\sKMW_0 \to \sKM_0= \Z$. This implies that $(U\Z)^- = 0$, just like $(H\Z)^- = 0$. On the other hand $\Z^+ \in \DMteZ^+ \simeq \mathrm{DM}^{\mathrm{eff}}(k,\Z[1/2])$ \cite[\S 5]{Deglise16} is the unit, by construction, whence $U\Z^+ = H\Z[1/2]$.
\end{proof}

\begin{exem}[Grayson's Motivic Cohomology]
In \cite{Suslin03}, Suslin proves that Grayson's definition of motivic cohomology coincides with Voevodsky's. To do so he proves that Grayson's complexes satisfy the cancellation theorem, and then employs an induction using poly-relative cohomology. We cannot resist pointing out that the second half of this argument is subsumed by our criterion. Indeed, it is easy to see that $K_0^\oplus$-presheaves admit framed transfers in the sense of \cite[\S 2]{Garkusha14}. Consequently the $\A^1$-localization functor for Grayson motives is given by $\mathrm{L}_{\Nis} \Cstar$ (\cite[Theorem 1.1]{Garkusha15}). Arguing exactly as in the proof of Corollary \ref{coro:effective} (using \cite[Remark 2.3]{Suslin03} instead of Proposition \ref{prop:ratcontractible}) we conclude that the spectrum $H\Z^{Gr}$ representing Grayson's motivic cohomology is effective. But $\Z^{Gr}(0) \simeq \Z$ and so $H\Z \simeq H\Z^{Gr}$, arguing as in the proof of Theorem \ref{thm:comparison}.
\end{exem}

\begin{exem}[GW-motives]
In \cite{Druzhinin17a}, there is defined a category of GW-motives $\mathrm{DM}^{\mathrm{GW}}(k)$, and the usual properties are established. Arguing very similarly to the proof of Proposition \ref{prop:ratcontractible}, one may show that the reduced GW-presheaf corresponding to $\gm^{\times n}$ is rationally contractible. Then, arguing as in Theorem \ref{thm:comparison} and Lemma \ref{lemm:modules}, using the main results of \cite{Druzhinin17a, Druzhinin17b, Druzhinin17c}, one may show that the spectrum representing $\1 \in \mathrm{DM}^{\mathrm{GW}}$ is $H\tilde{\Z}$ again, and that $\mathrm{DM}^{\mathrm{GW}}(k)$ is equivalent to the category of highly structured modules over $\H\tilde{\Z}$. In particular $\mathrm{DM}^{\mathrm{GW}}(k) \simeq \DMt$. We leave the details for further work.
\end{exem}

\begin{rem}
The assumption that $k$ is infinite in our results can be dropped by employing the techniques of \cite[Appendix B]{elmanto2017motivic}.
\end{rem}


\bibliography{General}{}

\begin{thebibliography}{10}

\bibitem{bachmann-hurewicz}
T.~Bachmann.
\newblock On the conservativity of the functor assigning to a motivic spectrum
  its motive.
\newblock 2016.
\newblock \href{https://arxiv.org/abs/1506.07375}{arXiv:1506.07375}.

\bibitem{bachmann-very-effective}
T.~Bachmann.
\newblock The generalized slices of hermitian k-theory.
\newblock {\em Journal of Topology}, 10(4):1124--1144, 2017.
\newblock \href{https://arxiv.org/abs/1610.01346}{arXiv:1610.01346}.

\bibitem{bachmann-real-etale}
T.~Bachmann.
\newblock Motivic and real etale stable homotopy theory.
\newblock {\em Accepted for publication in Compositio Mathematica}, 2017.
\newblock \href{https://arxiv.org/abs/1608.08855}{arXiv:1608.08855}.

\bibitem{Calmes14b}
B.~Calm\`es and J.~Fasel.
\newblock The category of finite {C}how-{W}itt correspondences.
\newblock \href{https://arxiv.org/abs/1412.2989}{arXiv:1412.2989}, 2014.

\bibitem{Cisinski12}
D.-C. Cisinski and F.~D{\'e}glise.
\newblock Triangulated categories of mixed motives.
\newblock arXiv:0912.2110, version 3, 2012.

\bibitem{Deglise16}
F.~D{\'e}glise and J.~Fasel.
\newblock {MW}-motivic complexes.
\newblock \href{https://arxiv.org/abs/1708.06095}{arXiv:1708.06095}, 2016.

\bibitem{Deglise17}
F.~D{\'e}glise and J.~Fasel.
\newblock The {M}ilnor-{W}itt motivic ring spectrum and its associated
  theories.
\newblock \href{https://arxiv.org/abs/1708.06102}{arXiv:1708.06102}, 2017.

\bibitem{Druzhinin17c}
A.~Druzhinin.
\newblock Cancellation theorem for {G}rothendieck-{W}itt-correspondences and
  {W}itt-correspondences.
\newblock \href{https://arxiv.org/abs/1709.06543}{arXiv:1709.06543}, 2017.

\bibitem{Druzhinin17a}
A.~Druzhinin.
\newblock Effective {G}rothendieck-{W}itt motives of smooth varieties.
\newblock \href{https://arxiv.org/abs/1709.06273}{arXiv:1709.06273}, 2017.

\bibitem{Druzhinin17b}
A.~Druzhinin.
\newblock Strict homotopy invariance of {N}isnevich sheaves with
  {GW}-transfers.
\newblock \href{https://arxiv.org/abs/1709.05805}{arXiv:1709.05805}, 2017.

\bibitem{Elmanto17}
E.~Elmanto and H.~A. Kolderup.
\newblock Modules over {M}ilnor-{W}itt motivic cohomology.
\newblock \href{https://arxiv.org/abs/1708.05651}{arxiv:1708.05651}, 2017.

\bibitem{elmanto2017motivic}
Elden Elmanto, Marc Hoyois, Adeel~A Khan, Vladimir Sosnilo, and Maria Yakerson.
\newblock Motivic infinite loop spaces.
\newblock {\em arXiv preprint arXiv:1711.05248}, 2017.

\bibitem{Fasel16b}
J.~Fasel and P.~A. {\O}stv{\ae}r.
\newblock A cancellation theorem for {MW}-motives.
\newblock \href{https://arxiv.org/abs/1708.06098}{arXiv:1708.06098}, 2016.

\bibitem{Garkusha17}
G.~Garkusha.
\newblock Reconstructing rational stable motivic homotopy theory.
\newblock \href{https://arxiv.org/abs/1705.01635}{arXiv:1705.01635}, 2017.

\bibitem{Garkusha14}
G.~Garkusha and I.~Panin.
\newblock Framed {M}otives of algebraic varieties.
\newblock \href{https://arxiv.org/abs/1409.4372}{arXiv:1409.4372}, 2014.

\bibitem{Garkusha15}
G.~Garkusha and I.~Panin.
\newblock Homotopy invariant presheaves with framed transfers.
\newblock \href{https://arxiv.org/abs/1504.00884}{arxiv:1504.00884}, 2015.

\bibitem{Grothendieck57}
A.~Grothendieck.
\newblock Sur quelques points d'alg\`ebre homologique.
\newblock {\em Tohoku}, 1(1):12--13, 1957.

\bibitem{EGAIV3}
A.~Grothendieck.
\newblock \'{E}l\'ements de g\'eom\'etrie alg\`ebrique: {IV}. \'{E}tude locale
  des sch\'emas et des morphismes de sch\'emas, {T}roisi\`eme partie.
\newblock {\em Publ. Math. Inst. Hautes \'Etudes Sci.}, 28:5--255, 1966.

\bibitem{Kolderup17}
H.~A. Kolderup.
\newblock Homotopy invariance of presheaves with {M}ilnor-{W}itt transfers.
\newblock \href{https://arxiv.org/abs/1708.04229}{arXiv:1708.04229}, 2017.

\bibitem{Levine06}
M.~Levine.
\newblock Chow's moving lemma and the homotopy coniveau tower.
\newblock {\em K-{T}heory}, 37(1):129--209, 2006.

\bibitem{levine2008homotopy}
M.~Levine.
\newblock The homotopy coniveau tower.
\newblock {\em Journal of Topology}, 1(1):217--267, 2008.

\bibitem{levine2013algebraic}
M.~Levine, Y.~Yang, and G.~Zhao.
\newblock Algebraic elliptic cohomology theory and flops, i.
\newblock {\em accepted for publication in Math. Annalen}, 2013.
\newblock With an appendix by Joel Riou.

\bibitem{mathew2017nilpotence}
A.~Mathew, N.~Naumann, and J.~Noel.
\newblock Nilpotence and descent in equivariant stable homotopy theory.
\newblock {\em Advances in Mathematics}, 305:994--1084, 2017.

\bibitem{Morel08}
F.~Morel.
\newblock {\em $\mathbb {A}^1$-{A}lgebraic {T}opology over a {F}ield}, volume
  2052 of {\em Lecture Notes in Math.}
\newblock Springer, New York, 2012.

\bibitem{Neeman}
A.~Neeman.
\newblock {\em Triangulated categories}, volume 148 of {\em Ann. of Math.
  Stud.}
\newblock Princeton Univ. Press, Princeton, N.J., 2001.

\bibitem{Neshitov14}
A.~Neshitov.
\newblock Framed correspondences and {M}ilnor-{W}itt {$K$}-theory.
\newblock arXiv:1410.7417. To appear in J. Inst. Math. Jussieu, 2014.

\bibitem{Panin09}
I.~Panin, A.~Stavrova, and N.~Vavilov.
\newblock Grothendieck-{S}erre conjecture {I}: {A}ppendix.
\newblock \href{https://arxiv.org/abs/0910.5465}{arXiv:0910.5465}, 2009.

\bibitem{Suslin03}
A.~A. Suslin.
\newblock On the {G}rayson spectral sequence.
\newblock {\em Tr. Mat. Inst. Steklova}, 241:218--253, 2003.

\bibitem{voevodsky-slice-filtration}
V.~Voevodsky.
\newblock Open problems in the motivic stable homotopy theory , i.
\newblock In {\em International Press Conference on Motives, Polylogarithms and
  Hodge Theory}. International Press, 2002.

\end{thebibliography}
\bibliographystyle{plain}


\end{document}